\documentclass[reqno,12pt]{amsart}
\usepackage{amsmath}
\usepackage{amssymb}
\usepackage{subfigure}
\usepackage{graphicx}
\usepackage{mathrsfs}
\usepackage{bbm}

\numberwithin{equation}{section}

\newcommand\NoBlackBoxes{\global\overfullrule0pt}
\NoBlackBoxes

\newtheorem{definition}{Definition}[section]
\newtheorem{theorem}[definition]{Theorem}
\newtheorem{lemma}[definition]{Lemma}

\newcommand{\N}{{\mathbb N}}
\newcommand{\R}{{\mathbb R}}

\newcommand{\E}{{\mathbb E}}
\renewcommand{\P}{{\mathbb P}}

\newcommand{\cov}{\mathrm{Cov}}

\newcommand{\X}{\textbf{X}}
\newcommand{\tr}{\mathrm{tr}}
\newcommand{\nc}{\mathcal{N}\mathcal{P}\mathcal{P}(k)}
\newcommand{\cc}{\mathcal{C}\mathcal{P}\mathcal{P}(k)}

\newcommand{\grab}{\bigskip \noindent}

\title[A phase transition in Random Matrix Theory]{A phase transition for the limiting spectral density of random matrices}

\author{Olga Friesen}
\address[Olga Friesen]{Westf\"alische Wilhelms-Universit\"at M\"unster,
Fachbereich Mathematik,
Einsteinstra\ss e 62, 48149 M\"unster, Germany}
\email[Olga Friesen]{olga.friesen@uni-muenster.de}

\author[Matthias L\"owe]{Matthias L\"owe}
\address[Matthias L\"owe]{Westf\"alische Wilhelms-Universit\"at M\"unster,
Fachbereich Mathematik,
Einsteinstra\ss e 62, 48149 M\"unster, Germany}
\email[Matthias L\"owe]{maloewe@math.uni-muenster.de}

\thanks{The research of the first author was supported by Deutsche Forschungsgemeinschaft via SFB 878 at University of M\"unster.}

\date{\today}



\begin{document}

\begin{abstract}
We analyze the spectral distribution of symmetric random matrices
with correlated entries. While we assume that the diagonals of these random matrices are
stochastically independent, the elements of the diagonals are taken to be correlated.
Depending on the strength of correlation the limiting spectral distribution is either the famous semicircle
law or some other law, related to that derived for Toeplitz matrices by Bryc, Dembo and Jiang (2006).
\end{abstract}

\maketitle
\section{Introduction}
Historically, the theory of random matrices is fed by two sources. They 
were introduced in mathematical statistics by the seminal work of Wishart \cite{Wishart28}. 
On the other hand, Wigner used random matrices as a toy model for the energy levels and excitation spectra of heavy nuclei \cite{Wigner}. From these two roots random matrix theory has grown into an independent mathematical theory with applications in many areas of science. 

A central role in the study of random matrices with growing dimension is played by their eigenvalues. 
To introduce them let, for any $n\in\N$, $\left\{a_n(p,q), 1\leq p\leq q\leq n\right\}$ be a real valued random field.
Define the symmetric random $n\times n$ matrix $\X_n$ by
\begin{equation*}
 \X_n (q,p) = \X_n (p,q) = \frac{1}{\sqrt{n}} a_n(p,q), \qquad 1\leq p\leq q\leq n.
\end{equation*} 

We will denote the (real) eigenvalues of $ \textbf{X}_{n}$ by $\lambda_1^{(n)} \le \lambda_2^{(n)} \le \ldots \lambda_n^{(n)}$. Let $\mu_n$ be the empirical eigenvalue distribution, i.e.
\begin{equation*}
\mu_n = \frac{1}{n} \sum_{k=1}^n \delta_{\lambda_k^{(n)}}.
\end{equation*}

Wigner proved in his fundamental work \cite{Wigner} that, if the entries $a_n(p,q)$ are independent, normally distributed with mean 0, and have variance 1 for off-diagonal elements, and variance 2 on the diagonal, the empirical eigenvalue distribution $\mu_n$ converges weakly (in probability) to the so called semicircle distribution (or law), i.e. the probability distribution $\nu$ on $\mathbb{R}$ with density
$$
\nu(dx)= \frac 1{2\pi} \sqrt{4-x^2} \mathbbmss{1}_{|x|\le 2}.
$$

Quite some effort has been spent in investigating the universality of this result. Arnold \cite{Arnold} showed that
the convergence to the semicircle law is also true if one replaces the Gaussian distributed random variables by independent and identically distributed (i.i.d.) random variables with a finite fourth moment. Also the identical distribution may be replaced by some other assumptions (see e.g. \cite{erdoes_survey}). Recently, it was observed by Erd\"os et al. (\cite{ESY}) that the convergence of the spectral measure
towards the semicircle law holds in a local sense. More precisely, it can be proved that 
on intervals with width going to zero sufficiently slowly, the empirical eigenvalue distribution still converges to the semicircle distribution. 

This result therefore interpolates between the global and the local behavior of the eigenvalues in the bulk of the spectrum, which was rather recently proved to be universal as well in the so-called ''four-moment-theorem'' (\cite{taovu2}).
 
Other generalizations of Wigner's semicircle law concern
matrix ensembles with entries drawn according to weighted Haar measures on classical
(e.g., orthogonal, unitary, symplectic) groups. Such results are particularly
interesting since such random matrices also play a major role in non-commutative
probability (see e.g. \cite{alice_stflour}, or the very recommendable book Anderson, Guionnet, and Zeitouni \cite{agz}).

A slightly different approach to universality was taken in \cite{Schenker_Schulz-Baldes} and \cite{Diag}. In \cite{Diag} we study matrices with correlated entries. It is shown that, if the diagonals of $\X_n$ are independent and the correlation between elements along a diagonal decays sufficiently quickly, again the limiting spectral distribution is the semicircle law. 

Universality, however, does have its limitations. As was shown by Bryc et al. \cite{brycdembo} the limiting spectral distribution of large random Toeplitz or Hankel matrices is not the semicircle law. In fact, not much is known about the limiting measures, apart from their moments (which are the result of the proof by a moment method, a technique, that will also be employed by the present paper). 

The present note tries to explore the borderline between the weak correlations studied in \cite{Diag} and the strong correlations that lead to a limiting spectral distribution that is not of Wigner type. We will again assume that $\X_n$ has independent diagonals and we will see, which quantity determines whether the limiting measure of the empirical eigenvalue distribution is a semicircle law or not. A particularly nice example is borrowed from statistical mechanics. There the Curie-Weiss model is the easiest model of a ferromagnet. Here a magnetic substance has little atoms that carry a magnetic spin, that is either $+1$ or $-1$. These spins interact in cooperative way, the strength of the interaction being triggered by a parameter, the so-called inverse temperature. The model exhibits phase transition from paramagnetic to magnetic behavior (the standard reference for the Curie-Weiss model is \cite{Ellis}). We will see that this phase transition can be recovered on the level of the limiting spectral distribution of random matrices, if we fill their diagonals independently with the spins of Curie-Weiss models. For small interaction parameter, this limiting spectral distribution is the semicircle law, while for a large interaction parameter we obtain a distribution similar to the Toeplitz case.

The rest of this paper is organized as follows. Section 2 contains the technical assumptions we have to make together with the statement of our main result. Section 3 characterizes the various limiting distributions we obtain. Section 4 contains some interesting examples, while Section 5 is devoted to the proof of the main theorem.

\section{Main Result}

This section contains the general theorem that describes the various limiting spectral distributions for the matrices $\X_n$ introduced above. In order to be able to state the theorem we
will have to impose the following conditions on $\X_n$: \\

\begin{enumerate}
 \item[(C1)] $\mathbb{E}\left[a_n(p,q)\right]=0$, $\mathbb{E}\left[a_n(p,q)^{2}\right] = 1$ and
\begin{equation}
 m_k:=\sup_{n\in\mathbb{N}} \max_{1\leq p\leq q\leq n} \mathbb{E}\left[\left|a_n(p,q)\right|^{k}\right] < \infty, \quad k\in\mathbb{N}.
 \label{mok}
\end{equation}
 \item[(C2)] the diagonals of $\X_n$, i.e. the families $\left\{a_n(p,p+r), 1\leq p\leq n-r\right\}$, $0\leq r\leq n-1$, are independent,
 \item[(C3)] the covariance of two entries on the same diagonal depends only on $n$, i.e. for any $0\leq r\leq n-1$ and $1\leq p,q\leq n-r$, $p\neq q$, we can define
	\begin{equation*}
       	\cov(a_n(p,p+r),a_n(q,q+r)) =: c_n,
	\end{equation*}
 \item[(C4)] the limit $c:=\lim_{n\to\infty} c_n$ exists.
\end{enumerate}

\grab

With these notations and conditions we are able to formulate the
central result of this note.

\begin{theorem}
Assume that the symmetric random matrix $\X_n$ as defined above satisfies the conditions (C1), (C2), (C3) and (C4). Then, with probability $1$, the empirical spectral distribution $\mu_n$ of $\X_n$ converges weakly to a nonrandom probability distribution $\nu_c$ which does not depend on the distribution of the entries of $\X_n$.
\label{main}
\end{theorem}

\section{The Limiting Distribution $\nu_c$}
\label{measure}

Since the proof of Theorem \ref{main} relies on the so-called moment-method, we want to describe the limiting spectral distribution $\nu_c$ in terms of its moments. It is not surprising that $\nu_c$ is some combination of the semicircle distribution and the limiting distribution of Toeplitz matrices as described in \cite{brycdembo}. Indeed, $c=0$ covers the case of independent entries implying that $\nu_0$ is the semicircle law. On the other hand, considering symmetric Toeplitz matrices, we have $c=1$, and thus $\nu_1$ is the corresponding limiting distribution we want to introduce in the following (cf. \cite{brycdembo}). Therefore, we have to start with some notation. For any even $k\in\N$, let $\mathcal{PP}(k)$ denote the set of all \emph{pair partitions} $\pi$ of $\left\{1,\ldots,k\right\}$. If $i$ and $j$ are in the same block of $\pi$, we also write $i\sim_\pi j$. The measure $\nu_1$ can be defined with the help of \emph{Toeplitz volumes}. Thus, we associate to any partition $\pi\in\mathcal{PP}(k)$ the following system of equations in unknowns $x_0,\ldots,x_k$:
\begin{equation}
\begin{split}
x_{1}-x_{0}+x_{l_1}-x_{l_1-1} &= 0, \quad \text{if} \ 1\sim_{\pi} l_1, \\
x_{2}-x_{1}+x_{l_2}-x_{l_2-1} &= 0, \quad \text{if} \ 2\sim_{\pi} l_2, \\
& \vdots \\
x_{i}-x_{i-1}+x_{l_i}-x_{l_i-1} &= 0, \quad \text{if} \ i\sim_{\pi} l_i, \\
& \vdots \\
x_{k}-x_{k-1}+x_{l_{k}}-x_{l_{k}-1} &= 0, \quad \text{if} \ k\sim_{\pi} l_k.
\end{split}
\label{eqSystem}
\end{equation}

Since $\pi$ is a pair partition, we in fact have only $k/2$ equations although we have listed $k$. However, we have $k+1$ variables. If $\pi=\{\{i_1,j_1\},\ldots,\{i_{k/2},j_{k/2}\}\}$ with $i_l<j_l$ for any $l=1,\ldots,{k/2}$, we solve \eqref{eqSystem} for $x_{j_1},\ldots,x_{j_{k/2}}$, and leave the remaining variables undetermined. We further impose the condition that all variables $x_0,\ldots,x_k$ lie in the interval $I=[0,1]$. Solving the equations above in this way determines a cross section of the cube $I^{k/2+1}$. The volume of this will be denoted by $p_T(\pi)$.

Returning to the measure $\nu_1$, we can use the results in \cite{brycdembo} to see that all odd moments of $\nu_1$ are zero, and for any even $k\in\N$, the $k$-th moment is given by
\begin{equation*}
\int x^k d\nu_1(x) = \sum_{\pi\in\mathcal{PP}(k)} p_T(\pi).
\end{equation*}

The expression above is bounded by $(k-1)!!$. Hence, Carleman's condition is satisfied implying that the distribution $\nu_1$ is uniquely determined by its moments. Moreover, it has an unbounded support as verified in \cite{brycdembo}. To describe $\nu_c$ for general $c\in\R$, we need a further definition which was introduced in \cite{brycdembo} to analyze Markov matrices.

\begin{definition}
Let $k\in\N$ be even, and fix $\pi\in\mathcal{PP}(k)$. The \emph{height} $h(\pi)$ of $\pi$ is the number of elements $i\sim_\pi j$, $i<j$, such that either $j=i+1$ or the restriction of $\pi$ to $\{i+1,\ldots,j-1\}$ is a pair partition.
\label{height}
\end{definition}

Note that the property that the restriction of $\pi$ to $\{i+1,\ldots,j-1\}$ is a pair partition in particular requires that the distance $j-i-1\geq 1$ is even. To give an example how to calculate the height of a partition, take $\pi=\{\{1,6\},\{2,4\},\{3,5\}\}$. Considering the block $\{1,6\}$, we see that the restriction of $\pi$ to $\{2,3,4,5\}$ is a pair partition, namely $\{\{2,4\},\{3,5\}\}$. However, this is not true for both remaining blocks. Hence, $h(\pi)=1$.

In Section~\ref{proof}, we will see that all odd moments of $\nu_c$ vanish, and the even moments are given by
\begin{equation}
\int x^k d\nu_c(x) = C_{\frac{k}{2}} + \sum_{\pi\in\cc} p_T(\pi) c^{\frac{k}{2}-h(\pi)} = \sum_{\pi\in\mathcal{PP}(k)} p_T(\pi) c^{\frac{k}{2}-h(\pi)},
\label{nu_c}
\end{equation}

where $C_k = \frac{(2k)!}{k!\left(k+1\right)!}$ denotes the $k$-th Catalan number, and $\cc$ is the set of crossing pair partitions of $\{1,\ldots,k\}$. Here, we say that a pair partition $\pi$ is \emph{crossing} if there are indices $i<j<l<m$ with $i\sim_\pi l$ and $j\sim_\pi m$. Otherwise, we call $\pi$ \emph{non-crossing}. We will denote the set of all non-crossing pair partitions of $\{1,\ldots,k\}$ by $\nc$. Note that the number of elements in $\nc$ coincides with the Catalan number $C_{k/2}$. The latter is exactly the $k$-th moment of the semicircle distribution. As for the limiting distribution in the Toeplitz case, we can verify the Carleman condition to see that $\nu_c$ is uniquely determined by its moments.

\section{Examples}

In this section, we want to give some examples of processes satisfying the assumptions of Theorem~\ref{main}.

\subsection{Toeplitz Matrices} Consider a symmetric Toeplitz matrix. The limiting spectral distribution calculated in \cite{brycdembo} can be deduced from Theorem~\ref{main} as well. Indeed, assuming that the entries are centered with unit variance and have existing moments of any order, we see that all conditions are satisfied with $c=c_n=1$. Thus, we get
\begin{equation*}
\int x^k d\nu_1(x) = \left\{
\begin{aligned}
& C_{\frac{k}{2}} + \sum_{\pi\in\cc} p_T(\pi) = \sum_{\pi\in\mathcal{PP}(k)} p_T(\pi), && \text{if} \ k \ \text{is even}, \\
& 0, && \text{if} \ k \ \text{is odd}.
\end{aligned} \right.
\end{equation*}

\subsection{Exchangeable Random Variables} Suppose that for any $n\in\N$, we have a family $\left\{x_n(p), 1\leq p\leq n\right\}$ of exchangeable random variables, i.e. the distribution of the vector $(x_n(1),\ldots,x_n(n))$ is the same as that of $(x_n(\sigma(1)),\ldots,x_n(\sigma(n)))$ for any permutation $\sigma$ of $\{1,\ldots,n\}$. In this case, we can conclude that for any $1\leq p<q\leq n$, we have
\begin{equation*}
\cov(x_n(p),x_n(q))=\cov(x_n(1),x_n(2))=:c_n.
\end{equation*}

Now assume that $c_n\to c\in\R$ as $n\to\infty$. Define for any $n\in\N$, $r\in\{0,\ldots,n-1\}$, the process $\{a_n(p,p+r), 1\leq p\leq n-r\}$ to be an independent copy of $\{x_n(p), 1\leq p\leq n-r\}$. Then, all conditions of Theorem~\ref{main} are satisfied if we ensure that the moment condition (C1) holds. The resulting limiting distribution for different choices of $c$ is depicted in Figure~\ref{exchange}.

\begin{figure}[ht]
  \centering
  \subfigure[$c=0.25$]{\includegraphics[width=5cm]{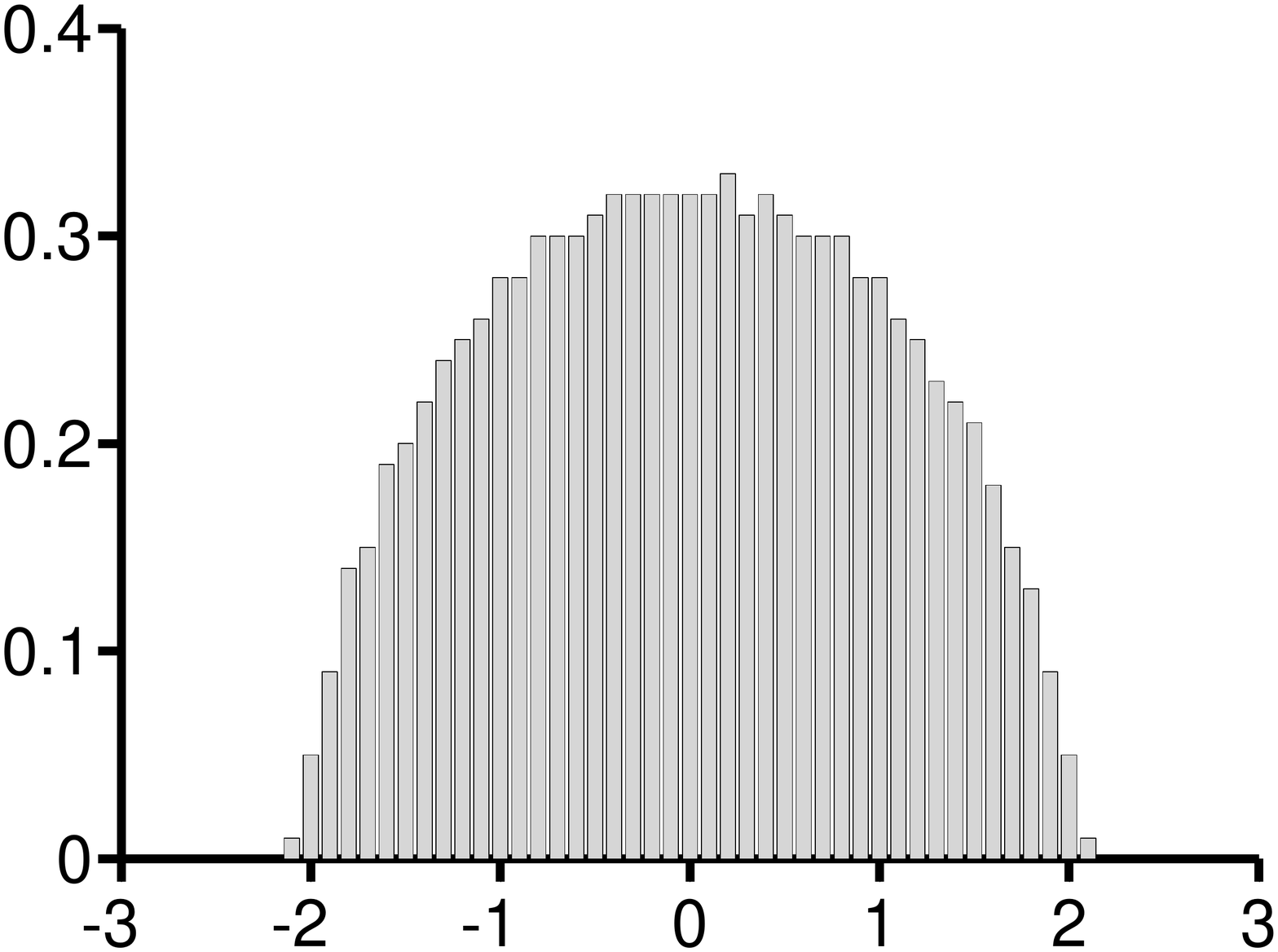}}
	\subfigure[$c=0.5$]{\includegraphics[width=5cm]{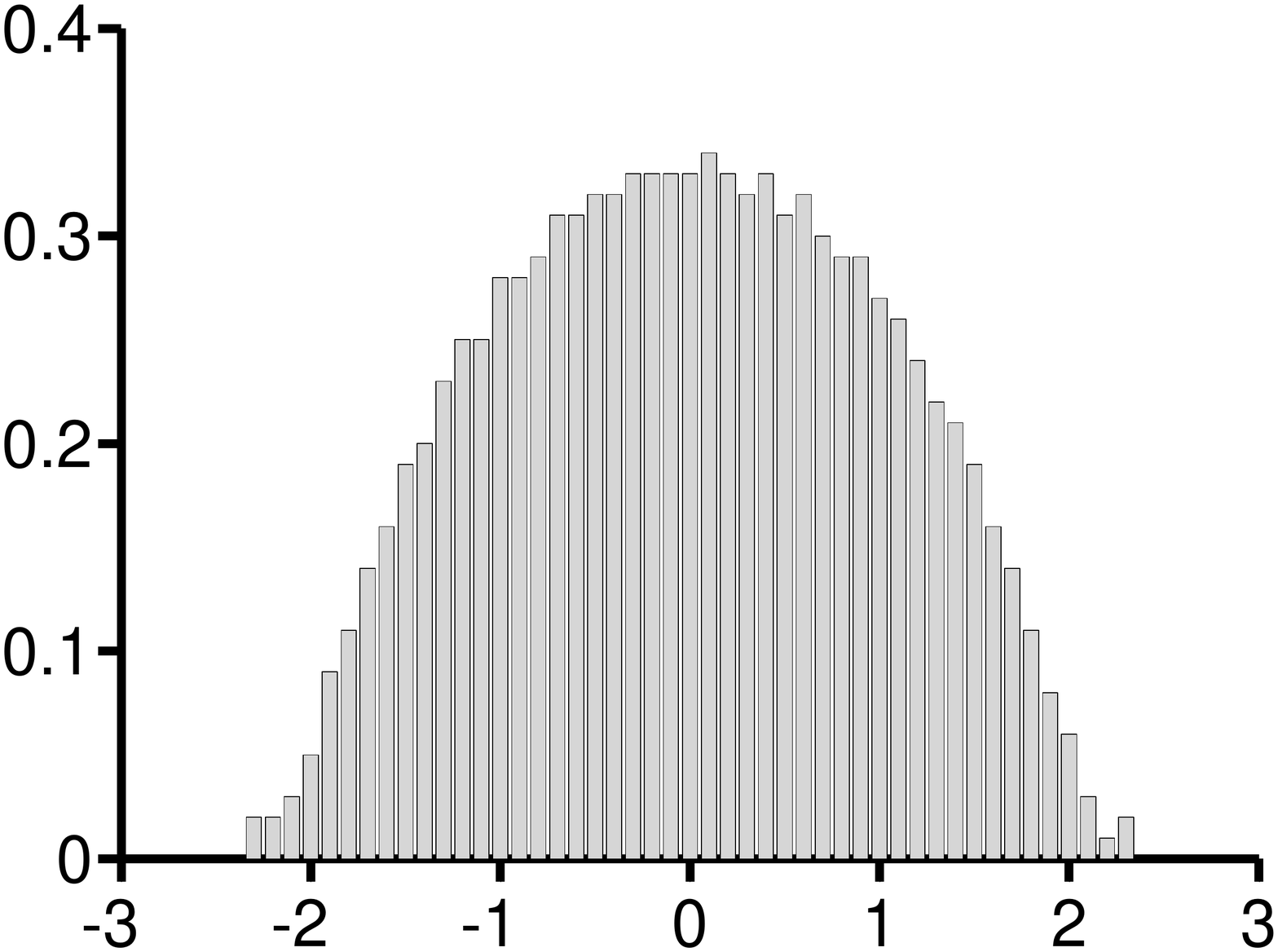}}
	\subfigure[$c=0.75$]{\includegraphics[width=5cm]{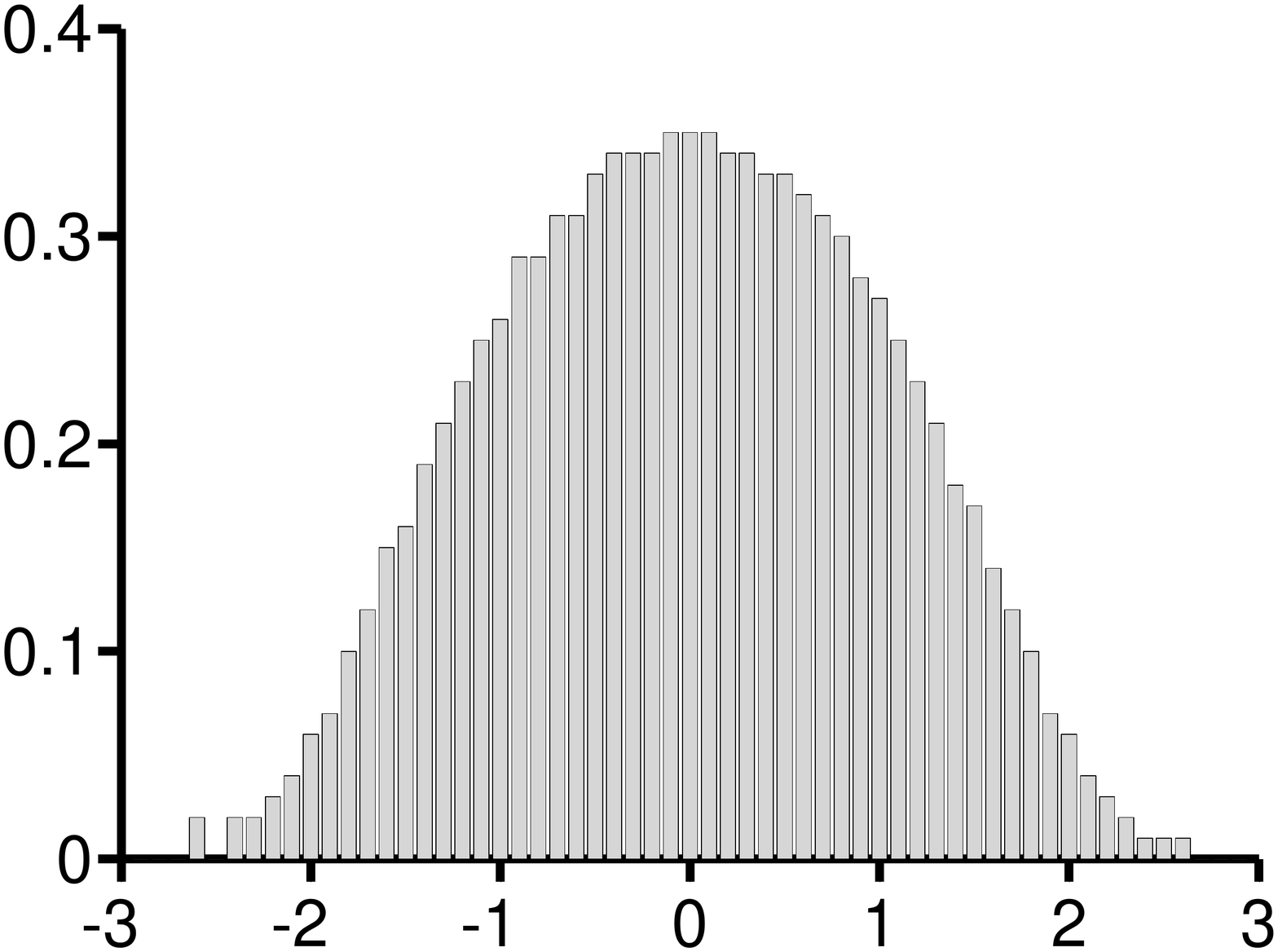}}
  \caption{Histograms of the empirical spectral distribution of $100$ realizations of $1000\times 1000$ matrices $\X_{1000}$ with standard Gaussian entries.}
  \label{exchange}
\end{figure}

An example for a process with exchangeable variables is the Curie-Weiss model with inverse temperature $\beta>0$. Here, the vector $x_n=(x_n(1),\ldots,x_n(n))$ takes values in $\{-1,1\}^n$, and for any $\omega=(\omega(1),\ldots,\omega(n))\in\{-1,1\}^n$, we have
\begin{equation*}
\P(x_n=\omega)
= \frac{1}{Z_{n,\beta}} \exp\left(\frac{\beta}{2n} \left(\sum_{i=1}^n \omega(i)\right)^2\right),
\end{equation*}

where $Z_{n,\beta}$ is the normalizing constant. Since $\P(x_n(1)=-1)=\P(x_n(1)=1)=\frac{1}{2}$, we obtain $\E[x_n(1)]=0$. Further, we clearly have $\E[x_n(1)^2]=1$. It remains to determine $c=\lim_{n\to\infty} c_n$. Therefore, we want to make use of the identity
\begin{equation*}
c_n = \cov(x_n(1),x_n(2)) = \E[x_n(1)x_n(2)] = \frac{n}{n-1} \E[m_n^2] - \frac{1}{n-1},
\end{equation*}

where $m_n:=\frac{1}{n}\sum_{i=1}^n x_n(i)$ is the so-called magnetization of the system. Since $|m_n|\leq 1$, we see that $m_n$ is uniformly integrable. Thus, $m_n$ converges in $\mathscr{L}^2$ to some random variable $m$ if and only if $m_n\to m$ in probability. In \cite{ellis2}, it was verified that $m_n\to 0$ in probability if $\beta\leq 1$, and $m_n\to m$ with $m\sim \frac{1}{2}\delta_{m(\beta)} + \frac{1}{2}\delta_{-m(\beta)}$ for some $m(\beta)>0$ if $\beta>1$. The function $m(\beta)$ is monotonically increasing on $(1,\infty)$, and satisfies $m(\beta)\to 0$ as $\beta \searrow 1$ and $m(\beta)\to 1$ as $\beta\to\infty$. We now obtain
\begin{equation*}
c = \lim_{n\to\infty} c_n = \left\{
\begin{aligned}
&0, && \text{if} \ \beta\leq 1, \\
&m(\beta)^2, && \text{if} \ \beta>1.
\end{aligned} \right.
\end{equation*}

Thus, the limiting spectral distribution of $\X_n$ is the semicircle law if $\beta\leq 1$, and approximately the Toeplitz limit if $\beta$ is large. This is insofar not surprising as the different sites in the Curie-Weiss model show little interaction, i.e. behave almost independently, if the temperature is high, or, in other words, $\beta$ is small. However, if the temperature is low, i.e. $\beta$ is large, the magnetization of the sites strongly depends on each other. The phase transition at the critical inverse temperature $\beta=1$ in the Curie-Weiss model is thus reflected in the limiting spectral distribution of $\X_n$ as well.

\section{Proof of Theorem \ref{main}}
\label{proof}

The main technique we want to apply is the method of moments. The idea is to first determine the weak limit of the expected empirical spectral distribution. Therefore, the similar structure of the matrices under consideration allows us to repeat some concepts presented in \cite{Diag}. However, we need to develop new ideas when calculating the expectations of the entries.

\subsection{The expected empirical spectral distribution}

To determine the limit of the $k$-th moment of the expected empirical spectral distribution $\mu_n$ of $\X_n$, we write
\begin{align*}
\E\left[\int x^k d\mu_n(x)\right]
&= \frac{1}{n} \E\left[\tr\left(\X_n^k\right)\right] \\
&=  \frac{1}{n^{\frac{k}{2}+1}} \sum_{p_1,\ldots,p_k=1}^n \E\left[a(p_1,p_2) a(p_2,p_3)\cdots a(p_{k-1},p_k) a(p_k,p_1)\right].
\end{align*}

The main task is now to compute the expectations on the right hand side. However, we have to face the problem that some of the entries involved are independent and some are not. To be more precise, $a(p_1,q_1),\ldots,a(p_j,q_j)$ are independent whenever they can be found on different diagonals of $\X_n$, i.e. the distances $|p_1-q_1|,\ldots,|p_j-q_j|$ are distinct. Hence, a first step in our proof is to consider the expectation $\E\left[a(p_1,p_2) a(p_2,p_3)\cdots a(p_{k-1},p_k) a(p_k,p_1)\right]$, and to identify entries with the same distance of their indices. Therefore, we want to adapt some concepts of \cite{Schenker_Schulz-Baldes} and \cite{brycdembo} to our situation.

To start with, fix $k\in\N$, and define $\mathcal{T}_n(k)$ to be the set of $k$-tuples of \emph{consistent pairs}, that is multi-indices $\left(P_1,\ldots,P_k\right)$ satisfying for any $j=1,\ldots,k$,

\begin{enumerate}
	\item $P_j = (p_j,q_j) \in \left\{1,\ldots,n\right\}^2$,
	\item $q_j = p_{j+1}$, where $k+1$ is cyclically identified with $1$.
\end{enumerate}

With this notation, we find that
\begin{equation*}
 \frac{1}{n} \E\left[\tr\left(\X_{n}^{k}\right)\right] = \frac{1}{n^{\frac{k}{2}+1}} \sum_{\left(P_1,\ldots,P_k\right)\in\mathcal{T}_n(k)} \E\left[a_n(P_1)\cdots a_n(P_k)\right].
\end{equation*}

To reflect the dependency structure among the entries $a_n(P_1)\ldots a_n(P_k)$, we want to make use of the set $\mathcal{P}(k)$ of partitions of $\{1,\ldots,k\}$. Thus, take $\pi\in\mathcal{P}(k)$. We say that an element $\left(P_1,\ldots,P_k\right)\in\mathcal{T}_n(k)$ is a \emph{$\pi$-consistent sequence} if
\begin{equation*}
 \left|p_i - q_i\right| = \left|p_j - q_j\right| \quad \Longleftrightarrow \quad i\sim_{\pi} j.
\end{equation*}

According to condition (C2), this implies that $a_n(P_{i_1}),\ldots,a_n(P_{i_l})$ are stochastically independent if $i_1,\ldots,i_l$ belong to $l$ different blocks of $\pi$. The set of all $\pi$-consistent sequences $\left(P_1,\ldots,P_k\right)\in\mathcal{T}_n(k)$ is denoted by $S_n(\pi)$. Note that the sets $S_n(\pi)$, $\pi\in\mathcal{P}(k)$, are pairwise disjoint, and $\bigcup_{\pi\in\mathcal{P}(k)} S_n(\pi)=\mathcal{T}_n(k)$. Consequently, we can write
\begin{equation}
 \frac{1}{n} \E\left[\tr\left(\X_{n}^{k}\right)\right] = \frac{1}{n^{\frac{k}{2}+1}} \sum_{\pi \in \mathcal{P}(k)} \sum_{\left(P_1,\ldots,P_k\right)\in S_n(\pi)} \E\left[a_n(P_1)\cdots a_n(P_k)\right].
\label{sum}
\end{equation}

In a next step, we want to exclude partitions that do not contribute to \eqref{sum} as $n\to\infty$. These are those partitions satisfying either $\#\pi > \frac{k}{2}$ or $\# \pi < \frac{k}{2}$, where $\#\pi$ denotes the number of blocks of $\pi$. We want to treat the two cases separately. \\

\emph{First case: $\#\pi > \frac{k}{2}$.} Since $\pi$ is a partition of $\left\{1,\ldots,k\right\}$, there is at least one singleton, i.e. a block containing only one element $i$. Consequently, $a_n(P_i)$ is independent of $\{a_n(P_j), j\neq i\}$ if $\left(P_1,\ldots,P_k\right)\in S_n(\pi)$. Since we assumed the entries to be centered, we obtain
\begin{equation*}
 \E\left[a_n(P_1)\cdots a_n(P_k)\right] = \E\Big[\prod_{i\neq l}a_n(P_i)\Big] \E\left[a_n(P_l)\right] = 0.
\end{equation*}

This yields
\begin{equation*}
 \frac{1}{n^{\frac{k}{2}+1}} \sum_{\left(P_1,\ldots,P_k\right)\in S_n(\pi)} \E\left[a_n(P_1)\cdots a_n(P_k)\right] = 0.
\end{equation*}

\emph{Second case: $r:= \#\pi < \frac{k}{2}$.} Here, we want to argue that $\pi$ gives vanishing contribution to \eqref{sum} as $n\to\infty$ by calculating $\#S_n(\pi)$. To fix an element $(P_1,\ldots,P_k)\in S_n(\pi)$, we first choose the pair $P_1 = (p_1,q_1)$. There are at most $n$ possibilities to assign a value to $p_1$, and another $n$ possibilities for $q_1$. To fix $P_2 = (p_2,q_2)$, note that the consistency of the pairs implies $p_2 = q_1$. If now $1\sim_\pi 2$, the condition $\left|p_1 - q_1\right| = \left|p_2 - q_2\right|$ allows at most two choices for $q_2$. Otherwise, if $1\not\sim_\pi 2$, we have at most $n$ possibilities. We now proceed sequentially to determine the remaining pairs. When arriving at some index $i$, we check whether $i$ is in the same block as some preceding index $1,\ldots,i-1$. If this is the case, then we have at most two choices for $P_i$ and otherwise, we have $n$. Since there are exactly $r=\#\pi$ different blocks, we can conclude that
\begin{equation}
 \# S_n(\pi) \leq n^2 n^{r-1} 2^{k-r} \leq C \ n^{r+1}
\label{secondSn}
\end{equation}
with a constant $C=C(r,k)$ depending on $r$ and $k$.

Now the uniform boundedness of the moments \eqref{mok} and the H\"{o}lder inequality together imply that for any sequence $(P_1,\ldots,P_k)$,
\begin{equation}
 \left|\E\left[a_n(P_1)\cdots a_n(P_k)\right] \right| \leq \left[\E\left|a_n(P_1)\right|^{k}\right]^{\frac{1}{k}} \cdots \left[\E\left|a_n(P_k)\right|^{k}\right]^{\frac{1}{k}} \leq m_k.
\label{holder}
\end{equation}

Consequently, taking account of the relation $r< \frac{k}{2}$, we get
\begin{align*}
 \frac{1}{n^{\frac{k}{2}+1}} \sum_{\left(P_1,\ldots,P_k\right)\in S_n(\pi)} \left|\E\left[a_n(P_1)\cdots a_n(P_k)\right] \right|
\leq C \ \frac{\# S_n(\pi)}{n^{\frac{k}{2}+1}}
\leq C \ \frac{1}{n^{\frac{k}{2}-r}} = o(1).
\end{align*}

Combining the calculations in the first and the second case, we can conclude that
\begin{equation*}
\frac{1}{n} \E\left[\tr\left(\X_{n}^{k}\right)\right]
= \frac{1}{n^{\frac{k}{2}+1}} \sum_{\substack{\pi \in \mathcal{P}(k), \\ \# \pi = \frac{k}{2}}} \sum_{\left(P_1,\ldots,P_k\right)\in S_n(\pi)} \E\left[a_n(P_1)\cdots a_n(P_k)\right] + o(1).
\end{equation*}

Now assume that $k$ is \emph{odd}. Then the condition $\#\pi = \frac{k}{2}$ cannot be satisfied, and the considerations above immediately yield
\begin{equation*}
 \lim_{n\to\infty} \frac{1}{n} \E\left[\tr\left(\X_{n}^{k}\right)\right] = 0.
\end{equation*}

It remains to determine the even moments. Thus, let $k\in\N$ be \emph{even}. Recall that we denoted by $\mathcal{PP}(k)\subset \mathcal{P}(k)$
the set of all pair partitions of $\{1,\ldots,k\}$. In particular, $\#\pi = \frac{k}{2}$ for any $\pi\in\mathcal{PP}(k)$.
On the other hand, if $\#\pi = \frac{k}{2}$ but $\pi \notin \mathcal{PP}(k)$, we can conclude that $\pi$ has at least one singleton and hence, as in the first case above, the expectation corresponding to the $\pi$-consistent sequences will become zero. Consequently,
\begin{equation}
 \frac{1}{n} \E\left[\tr\left(\X_{n}^{k}\right)\right] = \frac{1}{n^{\frac{k}{2}+1}} \sum_{\pi\in\mathcal{PP}(k)} \sum_{\left(P_1,\ldots,P_k\right)\in S_n(\pi)} \E\left[a_n(P_1)\cdots a_n(P_k)\right] + o(1).
 \label{sum2}
\end{equation}

We have now reduced the original set $\mathcal{P}(k)$ to the subset $\mathcal{P}\mathcal{P}(k)$. Next we want to fix a $\pi\in\mathcal{P}\mathcal{P}(k)$ and concentrate on the set $S_n(\pi)$. The following lemma will help us to calculate that part of \eqref{sum2} which involves non-crossing partitions.

\begin{lemma}[cf. \cite{brycdembo}, Proposition 4.4.]
 Let $S_{n}^{*}(\pi) \subseteq S_n(\pi)$ denote the set of $\pi$-consistent sequences $(P_1,\ldots,P_k)$ satisfying
\begin{equation*}
 i\sim_\pi j \quad \Longrightarrow \quad q_i - p_i = p_j - q_j
\end{equation*}
for all $i\neq j$. Then, we have
\begin{equation*}
 \# \left(S_n(\pi)\backslash S_{n}^{*}(\pi)\right) = o\left(n^{1+\frac{k}{2}}\right).
\end{equation*}
\label{snstar}
\end{lemma}

\begin{proof}
We call a pair $(P_i,P_j)$ with $i\sim_\pi j$, $i\neq j$, \emph{positive} if $q_i-p_i = q_j - p_j > 0$ and \emph{negative} if $q_i - p_i = q_j - p_j < 0$. Since $\sum_{i=1}^k q_i - p_i = 0$ by consistency, the existence of a negative pair implies the existence of a positive one. Thus, we can assume that any $(P_1,\ldots,P_k) \in S_n(\pi)\backslash S_{n}^{*}(\pi)$ contains a positive pair $(P_l,P_m)$. To fix such a sequence, we first determine the positions of $l$ and $m$, and then fix the signs of the remaining differences $q_i - p_i$. The number of possibilities to accomplish this depends only on $k$ and not on $n$. Now we choose one of $n$ possible values for $p_l$, and continue with assigning values to the differences $|q_i - p_i|$ for all $P_i$ except for $P_l$ and $P_m$. Since $\pi$ is a pair partition, we have at most $n^{\frac{k}{2}-1}$ possibilities for that. Then, $\sum_{i=1}^k q_i - p_i = 0$ implies that
\begin{equation*}
 0 < 2(q_l - p_l) = q_l - p_l + q_m - p_m = \sum_{\substack{i\in\{1,\ldots,k\}, \\ i\neq l,m}} p_i - q_i.
\end{equation*}
Since we have already chosen the signs of the differences $|q_i-p_i|$, $i\neq l,m$, as well as their absolute values, we know the value of the sum on the right hand side. Hence, the difference $q_l - p_l = q_m - p_m$ is fixed. We now have the index $p_l$, all differences $\left|q_i - p_i\right|, i\in\left\{1,\ldots,k\right\}$, and their signs. Thus, we can start at $P_l$ and go systematically through the whole sequence $(P_1,\ldots,P_k)$ to see that it is uniquely determined. Consequently, our considerations lead to
\begin{equation*}
 \# \left(S_n(\pi)\backslash S_{n}^{*}(\pi)\right) \leq C n^{\frac{k}{2}} = o\left(n^{1+\frac{k}{2}}\right).
\end{equation*}
\end{proof}

A consequence of Lemma~\ref{snstar} and relation \eqref{holder} is the identity
\begin{equation}
\frac{1}{n} \E\left[\tr\left(\X_{n}^{k}\right)\right]
= \frac{1}{n^{\frac{k}{2}+1}} \sum_{\pi\in\mathcal{PP}(k)} \sum_{\left(P_1,\ldots,P_k\right)\in S_{n}^{*}(\pi)} \E\left[a_n(P_1)\cdots a_n(P_k)\right] + o(1).
\label{altid}
\end{equation}

As already mentioned, the sets $S_n^*(\pi)$ help us to deal with the set $\nc$ of non-crossing pair partitions.

\begin{lemma}
 Let $\pi \in \nc$. For any $\left(P_1,\ldots,P_k\right)\in S_{n}^{*}(\pi)$, we have
\begin{equation*}
 \E\left[a_n(P_1)\cdots a_n(P_k)\right] = 1.
\end{equation*}
\label{noncrlemma}
\end{lemma}

\begin{proof}
 Let $l<m$ with $l\sim_\pi m$. Since $\pi$ is non-crossing, the number $l-m-1$ of elements between $l$ and $m$ must be even. In particular, there is $l\leq i< j\leq m$ with $i\sim_\pi j$ and $j=i+1$. By the properties of $S_{n}^{*}(\pi)$, we have $a_n(P_i)=a_n(P_j)$, and the sequence $\left(P_1,\ldots, P_l,\ldots,P_{i-1},P_{i+2},\ldots,P_m,\ldots,P_k\right)$ is still consistent. Applying this argument successively, all pairs between $l$ and $m$ vanish and we see that the sequence $\left(P_1,\ldots,P_l,P_m,\ldots,P_k\right)$ is consistent, that is $q_l=p_m$. Then, the identity $p_l=q_m$ also holds. In particular, $a_n(P_l)=a_n(P_m)$. Since this argument applies for arbitrary $l\sim_\pi m$, we obtain
\begin{equation*}
 \E\left[a_n(P_1)\cdots a_n(P_k)\right] = \prod_{\substack{l< m, \\ l\sim_\pi m}} \E\left[a_n(P_l) a_n(P_m)\right] = 1.
\end{equation*}
\end{proof}

By Lemma~\ref{noncrlemma}, we can conclude that
\begin{equation*}
\frac{1}{n^{\frac{k}{2}+1}} \sum_{\pi\in\nc} \sum_{\left(P_1,\ldots,P_k\right)\in S_{n}^{*}(\pi)} \E\left[a_n(P_1)\cdots a_n(P_k)\right]
= \frac{1}{n^{\frac{k}{2}+1}} \sum_{\pi\in\nc} \# S_n^*(\pi).
\end{equation*}

The following lemma allows us to finally calculate the term on the right hand side.

\begin{lemma}
 For any $\pi\in\nc$, we have
\begin{equation*}
 \lim_{n\to\infty} \frac{\# S_{n}^{*}(\pi)}{n^{\frac{k}{2}+1}} = 1.
\end{equation*}
\label{le2}
\end{lemma}

\begin{proof}
Since $\pi$ is non-crossing, we can find a nearest neighbor pair $i\sim_\pi i+1$. Now fix $(P_1,\ldots,P_k)\in S_n^*(\pi)$, and write $P_l=(p_l,p_{l+1})$, $l=1,\ldots,k$, where $k+1$ is identified with $1$. Then the properties of $S_n^*(\pi)$ ensure that $(p_i,p_{i+1})=(p_{i+2},p_{i+1})$. Hence, we can eliminate the pairs $P_i,P_{i+1}$ to obtain a sequence $(P_1^{(1)},\ldots,P_{k-2}^{(1)}):=(P_1,\ldots,P_{i-1},P_{i+2},\ldots,P_k)$ which is still consistent. Denote by $\pi'$ the partition obtained from $\pi$ by deleting the block $\{i,i+1\}$, and relabeling any $l\geq i+2$ to $l-2$. Since $\pi$ is non-crossing, we have $\pi'\in\mathcal{NPP}(k-2)$. Moreover, $(P_1^{(1)},\ldots,P_{k-2}^{(1)})\in S_n^*(\pi')$. Thus we see that any $(P_1,\ldots,P_k)\in S_n^*(\pi)$ can be reconstructed from a tuple $(P_1^{(1)},\ldots,P_{k-2}^{(1)})\in S_n^*(\pi')$ and a choice of $p_{i+1}$. The latter admits $n-\frac{k-2}{2}$ possibilities since $\{i,i+1\}$ forms a block on its own in $\pi$. Consequently,
\begin{equation}
\frac{\# S_{n}^{*}(\pi)}{n^{\frac{k}{2}+1}} = \frac{\# S_{n}^{*}(\pi')}{n^{\frac{k}{2}}} + o(1).
\label{induct}
\end{equation}

Now if $k=2$, we get $S_{n}^{*}(\pi)=\{((p,q),(q,p)): p,q\in\{1,\ldots,n\}\}$, implying $\frac{\# S_{n}^{*}(\pi)}{n^2}=1$. For arbitrary even $k\in\N$, the statement of Lemma \ref{le2} follows then by induction using the identity in \eqref{induct}.
\end{proof}

Taking account of the relation $\#\nc=C_{\frac{k}{2}}$, we now arrive at
\begin{multline}
\frac{1}{n} \E\left[\tr\left(\X_{n}^{k}\right)\right] \\
= C_{\frac{k}{2}} + \frac{1}{n^{\frac{k}{2}+1}} \sum_{\pi\in\cc} \sum_{\left(P_1,\ldots,P_k\right)\in S_n^*(\pi)} \E\left[a_n(P_1)\cdots a_n(P_k)\right] + o(1),
\label{sum3}
\end{multline}

with $\cc$ being the set of all crossing pair partitions of $\{1,\ldots,k\}$. Since we consider only pair partitions, we know that the expectation on the right hand side is of the form
\begin{equation*}
\E\left[a_n(p_1,q_1)a_n(p_1+\tau_1,q_1+\tau_1)\right]\cdots\E\left[a_n(p_r,q_r)a_n(p_r+\tau_r,q_r+\tau_r)\right],
\end{equation*}

for $r:=\frac{k}{2}$ and some choices of $p_1,q_1,\tau_1,\ldots,p_r,q_r,\tau_r\in\N$. In order to calculate this expectation, assumption (C3) indicates that we only need to distinguish for any $i=1,\ldots,k$, whether we have $\tau_i=0$ or not. In the first case, we get the identity $\E\left[a_n(p_i,q_i)a_n(p_i+\tau_i,q_i+\tau_i)\right] = 1$, in the second we can conclude that $\E\left[a_n(p_i,q_i)a_n(p_i+\tau_i,q_i+\tau_i)\right] = c_n$. Fix some pair partition $\pi\in\mathcal{PP}(k)$, and take $(P_1,\ldots,P_{k})\in S_{n}^{*}(\pi)$. Motivated by these considerations, we put
\begin{align*}
m\left(P_1,\ldots,P_k\right) := \#\{1\leq i<j\leq k:a_n(P_i)=a_n(P_j)\}.
\end{align*}

Obviously, we have $0\leq m\left(P_1,\ldots,P_k\right)\leq \frac{k}{2}$. With this notation, we find that
\begin{align}
\frac{1}{n^{\frac{k}{2}+1}} \sum_{\left(P_1,\ldots,P_{k}\right)\in S_{n}^{*}(\pi)} \E\left[a_n(P_1)\cdots a_n(P_{k})\right]
= \frac{1}{n^{\frac{k}{2}+1}} \sum_{l=0}^{k/2} c_n^{\frac{k}{2}-l} \# A_n^{(l)}\left(\pi\right),
\label{anl}
\end{align}

where
\begin{equation*}
A_n^{(l)}\left(\pi\right):=\{ \left(P_1,\ldots,P_k\right)\in S_{n}^{*}(\pi): m\left(P_1,\ldots,P_k\right) = l \}.
\end{equation*}

The following lemma states that if a pair $P_i,P_j$ contributes to $m(P_1,\ldots,P_k)$, then we can assume that the block $\{i,j\}$ in $\pi$ is not crossed by any other block.

\begin{lemma}
Let $\pi\in\mathcal{PP}(k)$ and fix $i\sim_\pi j$, $i<j$. Define
\begin{equation*}
S_n^*(\pi;i,j):= \{\left(P_1,\ldots,P_{k}\right)\in S_{n}^{*}(\pi): P_i=(p_i,q_i), P_j=(p_j,q_j), p_i=q_j, q_i=p_j\}.
\end{equation*}
Assume that there is some $i'\sim_\pi j'$ such that $i<i'<j$, and either $j'<i$ or $j<j'$. Then,
\begin{equation*}
\# S_n^*(\pi;i,j) = o\left(n^{\frac{k}{2}+1}\right).
\end{equation*}
\label{crenc}
\end{lemma}

\begin{proof}
To fix some $\left(P_1,\ldots,P_{k}\right)\in S_n^*(\pi;i,j)$, we first choose a value for $p_i=q_j$ and $q_i=p_j$. This allows for at most $n^2$ possibilities. Hence, $P_i$ and $P_j$ are fixed. Now consider the pairs $P_{i+1},\ldots,P_{i'-1}$. $p_{i+1}$ is uniquely determined by consistency. For $q_{i+1}$, there are at most $n$ choices. Then, $p_{i+2}=q_{i+1}$. If $i+2\sim_\pi i+1$, we have one choice for $q_{i+2}$. Otherwise, there are at most $n$. Proceeding in the same way, we see that we have $n$ possibilities whenever we start a new equivalence class. Similarly, we can assign values to the pairs $P_{j},\ldots,P_{i'+1}$ in this order. Now $P_{i'}$ is determined by consistency. When fixing $P_{i-1},\ldots,P_1,P_k,\ldots,P_{j+1}$, we again have $n$ choices for any new equivalence class. To sum up, we are left with at most
\begin{equation*}
n^2 n^{\frac{k}{2}-2} = n^{\frac{k}{2}}
\end{equation*}
possible values for an element in $S_n^*(\pi;i,j)$.
\end{proof}

Recall Definition \ref{height} where we introduced the notion of the \emph{height} $h(\pi)$ of a pair partition $\pi$. Lemma \ref{crenc} in particular implies that only those $\left(P_1,\ldots,P_k\right)\in S_n^*(\pi)$ with
\begin{equation*}
0 \leq m\left(P_1,\ldots,P_k\right) \leq h(\pi)
\end{equation*}

contribute to the limit of \eqref{anl}. Indeed, if $m(P_1,\ldots,P_k) > h(\pi)$, we can find some $i\sim_\pi j$, $i<j$, such that $(P_1,\ldots,P_k)\in S_n^*(\pi;i,j)$ and neither $j=i+1$ nor is the restriction of $\pi$ to $\{i+1,\ldots,j-1\}$ a pair partition. Hence, the crossing property in Lemma~\ref{crenc} is satisfied, and $(P_1,\ldots,P_k)$ is contained in a set that is negligible in the limit. The identity in \eqref{anl} thus becomes
\begin{equation*}
\frac{1}{n^{\frac{k}{2}+1}} \sum_{\left(P_1,\ldots,P_{k}\right)\in S_{n}^{*}(\pi)} \E\left[a_n(P_1)\cdots a_n(P_{k})\right]
= \frac{1}{n^{\frac{k}{2}+1}} \sum_{l=0}^{h(\pi)} c_n^{\frac{k}{2}-l} \# B_n^{(l)}(\pi) + o(1),
\end{equation*}

where
\begin{multline*}
B_n^{(l)}(\pi):=\left\{ \left(P_1,\ldots,P_k\right)\in S_{n}^{*}(\pi): m\left(P_1,\ldots,P_k\right) = l;\right. \\
\left. a_n(P_i)=a_n(P_j), i<j \ \Rightarrow \ j=i+1 \ \text{or} \ \pi|_{\{i+1,\ldots,j-1\}} \ \text{is a pair partition} \right\}.
\end{multline*}

In the next step, we want to simplify the expression above further by showing that $B_n^{(l)}(\pi)=\emptyset$ whenever $0\leq l<h(\pi)$. This is ensured by

\begin{lemma}
Let $\pi\in\mathcal{PP}(k)$. For any $(P_1,\ldots,P_k)\in S_{n}^{*}(\pi)$, we have
\begin{equation*}
m(P_1,\ldots,P_k)\geq h(\pi).
\end{equation*}
\label{hpi}
\end{lemma}

\begin{proof}
If $h(\pi)=0$, there is nothing to prove. Thus, suppose that $h(\pi)\geq 1$ and take some $i\sim_\pi j$, $i<j$, such that either $j=i+1$ or $j-i-1\geq 2$ is even and the restriction of $\pi$ to $\{i+1,\ldots,j-1\}$ is a pair partition. Fix $(P_1,\ldots,P_k)\in S_n^*(\pi)$, and write $P_l=(p_l,p_{l+1})$ for any $l=1,\ldots,k$. We need to verify that $p_{i+1}=p_j$. If we achieve this, the definition of $S_n^*(\pi)$ will also ensure that $p_i=p_{j+1}$. As a consequence, the $\pi$-block $\{i,j\}$ will contribute to $m(P_1,\ldots,P_k)$. Since there are $h(\pi)$ such blocks, we will obtain $m(P_1,\ldots,P_k)\geq h(\pi)$ for any choice of $(P_1,\ldots,P_k)\in S_n^*(\pi)$.

If $j=i+1$, we immediately obtain $p_{i+1}=p_j$. To show this property in the second case, note that the sequence $(P_{i+1},\ldots,P_{j-1})$ solves the following system of equations:
\begin{align*}
p_{i+2}-p_{i+1}+p_{l_1+1}-p_{l_1} &= 0, \quad \text{if} \ i+1\sim_\pi l_1, \\
p_{i+3}-p_{i+2}+p_{l_2+1}-p_{l_2} &= 0, \quad \text{if} \ i+2\sim_\pi l_2, \\
& \vdots \\
p_{i+m+1}-p_{i+m}+p_{l_m+1}-p_{l_m} &= 0, \quad \text{if} \ i+m\sim_\pi l_m, \\
& \vdots \\
p_{j}-p_{j-1}+p_{l_{j-i-1}+1}-p_{l_{j-i-1}} &= 0, \quad \text{if} \ j-1\sim_\pi l_{j-i-1}.
\end{align*}

Start with solving the first equation for $p_{i+2}$ which yields
\begin{equation*}
p_{i+2}=p_{i+1}-p_{l_1+1}+p_{l_1}.
\end{equation*}

Then, insert this in the second equation, and solve it for $p_{i+3}$ to obtain
\begin{equation*}
p_{i+3}=p_{i+1}-p_{l_1+1}+p_{l_1}-p_{l_2+1}+p_{l_2}.
\end{equation*}

In the $j-i-1$-th step, we substitute $p_{j-1}=p_{i+(j-i-1)}$ in the $j-i-1$-th equation, and solve it for $p_j=p_{i+(j-i-1)+1}$. We then have
\begin{equation*}
p_j=p_{i+1}-\sum_{m=1}^{j-i-1} (p_{l_m+1}-p_{l_m}).
\end{equation*}

Since the restriction of $\pi$ to $\{i+1,\ldots,j-1\}$ is a pair partition, we can conclude that the sets $\{l_1,\ldots,l_{j-i-1}\}$ and $\{i+1,\ldots,j-1\}$ are equal. Hence, we obtain $\sum_{m=1}^{j-i-1} (p_{l_m+1}-p_{l_m})=p_j-p_{i+1}$, implying $p_j=p_{i+1}$.

\end{proof}

With the help of Lemma \ref{hpi}, we thus arrive at
\begin{equation*}
\frac{1}{n^{\frac{k}{2}+1}} \sum_{(P_1,\ldots,P_{k})\in S_{n}^{*}(\pi)} \E\left[a_n(P_1)\cdots a_n(P_k)\right]
= \frac{\# B_n^{(h(\pi))}(\pi)}{n^{\frac{k}{2}+1}} \ c_n^{\frac{k}{2}-h(\pi)} + o(1).
\end{equation*}

Note that any element $(P_1,\ldots,P_{k})\in S_{n}^{*}(\pi)$ satisfying the condition
\begin{equation}
a_n(P_i)=a_n(P_j), \ i<j \quad \Rightarrow \quad j=i+1 \ \text{or} \ \pi|_{\{i+1,\ldots,j-1\}} \ \text{is a pair partition},
\label{cond_a}
\end{equation}

fulfills the condition $m(P_1,\ldots,P_k)= h(\pi)$ as well. Indeed, \eqref{cond_a} guarantees that $m(P_1,\ldots,P_k)\leq h(\pi)$, and Lemma \ref{hpi} ensures that $m(P_1,\ldots,P_k)\geq h(\pi)$. Thus, we can write
\begin{multline*}
B_n^{(h(\pi))}(\pi) =\left\{ \left(P_1,\ldots,P_k\right)\in S_{n}^{*}(\pi): \right. \\
\left. a_n(P_i)=a_n(P_j), i<j \ \Rightarrow \ j=i+1 \ \text{or} \ \pi|_{\{i+1,\ldots,j-1\}} \ \text{is a pair partition} \right\}.
\end{multline*}

Now any element in the complement of $B_n^{(h(\pi))}(\pi)$ satisfies for some $i\sim_\pi j$ the crossing assumption in Lemma \ref{crenc}. This yields
\begin{equation*}
\frac{\#\left(B_n^{(h(\pi))}(\pi)\right)^c}{n^{\frac{k}{2}+1}} = o(1).
\end{equation*}

Since $B_n^{(h(\pi))}(\pi)\cup\left(B_n^{(h(\pi))}(\pi)\right)^c = S_{n}^{*}(\pi)$, we obtain that
\begin{equation}
\frac{1}{n^{\frac{k}{2}+1}} \sum_{\left(P_1,\ldots,P_{k}\right)\in S_{n}^{*}(\pi)} \E\left[a_n(P_1)\cdots a_n(P_{k})\right]
= \frac{\# S_n^*(\pi)}{n^{\frac{k}{2}+1}} \ c_n^{\frac{k}{2}-h(\pi)} + o(1).
\label{eqToepVol}
\end{equation}

To calculate the limit on the right-hand side, we have

\begin{lemma}[cf. \cite{brycdembo}, Lemma 4.6]
For any $\pi\in\mathcal{PP}(k)$, it holds that
\begin{equation*}
\lim_{n\to\infty} \frac{\#S_n^*(\pi)}{n^{\frac{k}{2}+1}} = p_T(\pi),
\end{equation*}

where $p_T(\pi)$ is the Toeplitz volume defined by solving the system of equations \eqref{eqSystem}.
\label{dembo}
\end{lemma}

\begin{proof}
Fix $\pi\in\mathcal{PP}(k)$. Note that if $P=\{(p_i,p_{i+1}), i=1,\ldots,k\}\in S_n^*(\pi)$, then $x_0,x_1,\ldots,x_k$ with $x_i=p_{i+1}/n$ is a solution of the system of equations \eqref{eqSystem}. On the other hand, if $x_0,x_1,\ldots,x_k \in\{1/n,2/n,\ldots,1\}$ is a solution of \eqref{eqSystem} and $p_{i+1}=nx_i$, then either $\{(p_i,p_{i+1}), i=1,\ldots,k\}\in S_n^*(\pi)$ or $\{(p_i,p_{i+1}), i=1,\ldots,k\}\in S_n(\eta)$ for some partition $\eta\in\mathcal{P}(k)$ such that $i\sim_\pi j \Rightarrow i\sim_\eta j$, but $\#\eta<\#\pi$.

In \eqref{eqSystem}, we have $k+1$ variables and only $k/2$ equations. Denote the $k/2+1$ undetermined variables by $y_1,\ldots,y_{k/2+1}$. We thus need to assign values from the set $\{1/n,2/n,\ldots,1\}$ to $y_1,\ldots,y_{k/2+1}$, and then to calculate the remaining $k/2$ variables from the equations. Since the latter are also supposed to be in the range $\{1/n,2/n,\ldots,1\}$, it might happen that not all values for the undetermined variables are admissible. Let $p_n(\pi)$ denote the admissible fraction of the $n^{k/2+1}$ choices for $y_1,\ldots,y_{k/2+1}$. By our remark at the beginning of the proof and estimate \eqref{secondSn}, we have that
\begin{equation*}
\lim_{n\to\infty} \frac{\#S_n^*(\pi)}{n^{\frac{k}{2}+1}} = \lim_{n\to\infty} p_n(\pi),
\end{equation*}

if the limits exist. Now we can interpret $y_1,\ldots,y_{k/2+1}$ as independent random variables with a uniform distribution on $\{1/n,2/n,\ldots,1\}$. Then, $p_n(\pi)$ is the probability that the computed values stay within the interval $(0,1]$. As $n\to\infty$, $y_1,\ldots,y_{k/2+1}$ converge in law to independent random variables uniformly distributed on $[0,1]$. Hence, $p_n(\pi)\to p_T(\pi)$.
\end{proof}

Applying Lemma~\ref{dembo} and assumption (C4) to equation \eqref{eqToepVol}, we arrive at
\begin{equation*}
\lim_{n\to\infty}\frac{1}{n^{\frac{k}{2}+1}} \sum_{\left(P_1,\ldots,P_{k}\right)\in S_{n}^{*}(\pi)} \E\left[a_n(P_1)\cdots a_n(P_{k})\right]
= p_T(\pi) c^{\frac{k}{2}-h(\pi)}.
\end{equation*}

Substituting this result in \eqref{sum3}, we find that for any even $k\in\N$,
\begin{align*}
\lim_{n\to\infty} \frac{1}{n} \E\left[\tr\left(\X_n^k\right)\right]
= C_{\frac{k}{2}} + \sum_{\pi\in\cc} p_T(\pi) c^{\frac{k}{2}-h(\pi)}.
\end{align*}

To obtain the alternative expression in \eqref{nu_c} for the even moments of the limiting measure $\nu_c$, note that the considerations above were not restricted to crossing partitions. In particular, we can start from identity \eqref{altid} instead of \eqref{sum3} to see that
\begin{align*}
\lim_{n\to\infty} \frac{1}{n} \E\left[\tr\left(\X_n^k\right)\right]
= \lim_{n\to\infty} \sum_{\pi\in\mathcal{PP}(k)} \frac{\#S_n^*(\pi)}{n^{\frac{k}{2}+1}} \ c_n^{\frac{k}{2}-h(\pi)}
= \sum_{\pi\in\mathcal{PP}(k)} p_T(\pi) c^{\frac{k}{2}-h(\pi)}.
\end{align*}

\subsection{Almost Sure Convergence}

The almost sure convergence of the empirical distribution is a consequence of the following concentration inequality proven in \cite{brycdembo} and \cite{Diag}.

\begin{lemma}
Suppose that conditions (C1) and (C2) hold. Then, for any $k,n \in\N$,
\begin{equation*}
 \E\left[\left(\tr\left(\X_{n}^{k}\right) - \E\left[\tr \left(\X_{n}^{k}\right)\right]\right)^4\right] \leq C \ n^{2}.
\end{equation*}
\label{lemma}
\end{lemma}

From Lemma~\ref{lemma} and Chebyshev's inequality, we can now conclude that for any $\varepsilon>0$ and any $k,n\in\N$,
\begin{equation*}
\P\left( \left| \frac{1}{n} \tr\left(\X_{n}^{k}\right) - \E \left[\frac{1}{n}\tr\left(\X_{n}^{k}\right)\right] \right| > \varepsilon \right) \leq \frac{C}{\varepsilon^4 n^2}.
\end{equation*}

Applying the Borel-Cantelli lemma, we see that
\begin{equation}
\frac{1}{n} \tr\left(\X_{n}^{k}\right) - \E\left[\frac{1}{n}\tr\left(\X_{n}^{k}\right)\right] \to 0, \quad \text{a.s.}.
\label{aslim}
\end{equation}

Let $Y$ be a random variable distributed according to $\nu_c$. The convergence of the moments of the expected empirical distributions and relation \eqref{aslim} yield
\begin{equation*}
\frac{1}{n} \tr\left(\X_{n}^{k}\right) \to \E[Y^k], \quad \text{a.s.}.
\end{equation*}

Since the distribution of $Y$ is uniquely determined by its moments, we obtain almost sure weak convergence of the empirical spectral distribution of $\X_n$ to $\nu_c$.

\bibliographystyle{alpha}
\bibliography{gabi}

\begin{thebibliography}{AGZ10}

\bibitem[AGZ10]{agz}
Greg~W. Anderson, Alice Guionnet, and Ofer Zeitouni.
\newblock {\em An Introduction to Random Matrices}.
\newblock Cambridge studies in advanced mathematics 118. Cambridge University
  Press, Cambridge, 2010.

\bibitem[Arn71]{Arnold}
Ludwig Arnold.
\newblock On {W}igner's semicircle law for the eigenvalues of random matrices.
\newblock {\em Z. Wahrscheinlichkeitstheorie und Verw. Gebiete}, 19:191--198,
  1971.

\bibitem[BDJ06]{brycdembo}
W{\l}odzimierz Bryc, Amir Dembo, and Tiefeng Jiang.
\newblock Spectral measure of large random {H}ankel, {M}arkov and {T}oeplitz
  matrices.
\newblock {\em Ann. Probab.}, 34(1):1--38, 2006.

\bibitem[Ell85]{Ellis}
Richard~S. Ellis.
\newblock {\em Entropy, large deviations, and statistical mechanics}, volume
  271 of {\em Grundlehren der Mathematischen Wissenschaften [Fundamental
  Principles of Mathematical Sciences]}.
\newblock Springer-Verlag, New York, 1985.

\bibitem[EN78]{ellis2}
Richard Ellis and Charles Newman.
\newblock Fluctuationes in {C}urie-{W}eiss exemplis.
\newblock In {\em Mathematical Problems in Theoretical Physics}, volume~80 of
  {\em Lecture Notes in Physics}, pages 313--324. Springer Berlin / Heidelberg,
  1978.

\bibitem[Erd11]{erdoes_survey}
L{\'a}szl{\'o} Erd{\H{o}}s.
\newblock Universality of {W}igner random matrices: a survey of recent results.
\newblock {\em Uspekhi Mat. Nauk}, 66(3(399)):67--198, 2011.

\bibitem[ESY09]{ESY}
L{\'a}szl{\'o} Erd{\H{o}}s, Benjamin Schlein, and Horng-Tzer Yau.
\newblock Local semicircle law and complete delocalization for {W}igner random
  matrices.
\newblock {\em Comm. Math. Phys.}, 287(2):641--655, 2009.

\bibitem[FL11]{Diag}
Olga Friesen and Matthias L\"owe.
\newblock The semicircle law for matrices with independent diagonals.
\newblock {\em Preprint, to appear in J. Theoret. Probab.}, 2011.

\bibitem[Gui09]{alice_stflour}
Alice Guionnet.
\newblock {\em Large random matrices: lectures on macroscopic asymptotics},
  volume 1957 of {\em Lecture Notes in Mathematics}.
\newblock Springer-Verlag, Berlin, 2009.
\newblock Lectures from the 36th Probability Summer School held in Saint-Flour,
  2006.

\bibitem[SSB05]{Schenker_Schulz-Baldes}
Jeffrey Schenker and Hermann Schulz-Baldes.
\newblock Semicircle law and freeness for random matrices with symmetries or
  correlations.
\newblock {\em Math. Res. Lett.}, 12:531--542, 2005.

\bibitem[TV11]{taovu2}
Terence Tao and Van Vu.
\newblock Random matrices: Universality of local eigenvalue statistics.
\newblock {\em Acta Mathematica}, 206:127--204, 2011.

\bibitem[Wig58]{Wigner}
Eugene~P. Wigner.
\newblock On the distribution of the roots of certain symmetric matrices.
\newblock {\em Ann. of Math.}, 67:325--328, 1958.

\bibitem[Wis28]{Wishart28}
John Wishart.
\newblock The generalized product moment distribution in samples from a normal
  multivariate population.
\newblock {\em Biometrika}, 20:32--52, 1928.

\end{thebibliography}

\end{document}